\documentclass[a4paper,11pt]{article}

\usepackage[a4paper,left=26.5mm,right=26mm,top=26.25mm,bottom=36.75mm]{geometry}

\usepackage{authblk}

\title{
Luenberger observers for discrete-time nonlinear systems
}

\author{Lucas Brivadis, Vincent Andrieu and Ulysse Serres
}
\affil{The authors are with Univ. Lyon, Universit\'e Claude Bernard Lyon 1, CNRS, LAGEPP UMR 5007, 43 bd du 11 novembre 1918, F-69100 Villeurbanne, France
(e-mail: lucas.brivadis@gmail.com, vincent.andrieu@gmail.com, ulysse.serres@gmail.com)
}

\usepackage{amsfonts, amssymb, amsmath, amsthm}
\usepackage{mathrsfs}
\usepackage{hyperref}
\usepackage{epsfig} % for postscript graphics files

\newcommand{\R}{\mathbb{R}}
\newcommand{\N}{\mathbb{N}\cup\{0\}} 

\newcommand{\C}{\mathbb{C}}

\newcommand{\X}{\mathcal{X}} 
\newcommand{\K}{\mathcal{K}} 
\newcommand{\RR}{\mathcal{R}}
\newcommand{\DD}{\mathcal{D}}
\renewcommand{\SS}{\mathcal{S}}

\renewcommand{\leq}{\leqslant}
\renewcommand{\geq}{\geqslant}

\renewcommand{\epsilon}{\varepsilon}
\renewcommand{\phi}{\varphi}

\newcommand{\norm}[1]{\left\Vert #1\right\Vert}

\newcommand{\dd}{\mathrm{d}}

\DeclareMathOperator{\diag}{diag}

\newcommand{\ie}{\emph{i.e.~}}

\newcommand{\xx}{\underline{x}}
\newcommand{\dt}{{dt}}%{\mathrm{dt}}

\newtheorem{assumption}{Assumption}
\newtheorem{theorem}{Theorem}

\newtheorem{remark}{Remark}
\newtheorem{lemma}{Lemma}
\newtheorem{definition}{Definition}
\newtheorem{proposition}{Proposition}

\begin{document}

\maketitle

%%%%%%%%%%%%%%%%%%%%%%%%%%%%%%%%%%%%%%%%%%%%%%%%%%%%%%%%%%%%%%%%%%%%%%%%%%%%%%%%
\begin{abstract}
In this paper, we consider the problem of designing an asymptotic observer for a nonlinear dynamical system in discrete-time following Luenberger's original idea.
This approach is a two-step design procedure. In a first step, the problem is to estimate a function of the state.
The state estimation is obtained by inverting this mapping.
Similarly to the continuous-time context, we show that the first step is always possible provided a linear and stable discrete-time system fed by the output is introduced.
Based on a weak observability assumption, it is shown that picking the dimension of the stable auxiliary system sufficiently large, the estimated function of the state is invertible.
This approach is illustrated on linear systems with polynomial output.
The link with the Luenberger observer obtained in the continuous-time case is also investigated.
\end{abstract}

%%%%%%%%%%%%%%%%%%%%%%%%%%%%%%%%%%%%%%%%%%%%%%%%%%%%%%%%%%%%%%%%%%%%%%%%%%%%%%%%
\section{Introduction}

\subsection{Context}

The design of observers for nonlinear discrete-time systems remains a challenging and open problem despite a burgeoning literature.
Since no universal method exists, several approaches have been developed. Most of them have first been developed for continuous-time systems, and then extended to the discrete case.
Some of them, such as the well-known \emph{extended Kalman filter} (\cite{Kalman1, Kalman2}), provide only a local convergence of the observer, and are based on a linearization of the system.
Others (as \cite{outputinj1} or \cite{outputinj2}) consist in applying an invertible change of coordinates that transforms the original system in an other form for which it is much more easier to design an observer.
% But the existence of such a change of coordinates requires strong assumptions on the system such as analyticity of the system and observablity of the linearized system.
Still others deal with Lipschitz nonlinear systems (\cite{lip1, lip2}, among others), that occur frequently in practice, and are based on linear matrix inequalities that provide Lyapunov functions for the error system.

A completely different idea is to try to reproduce the Luenberger's initial methodology originally developed for linear continuous-time system in \cite{Luenberger}, which differs from what is now usually called \emph{Luenberger observer}.
This path has been mapped in the case of discrete-time systems by N. Kazantzis and C. Kravaris in \cite{KK1}.
It consists to estimate first a function of the state, thanks to a linear stable system fed by the output, and then to inverse this mapping.
However, strong assumptions such as analyticity of the system and observablity of the linearized system are required, and the invertibility of the function is obtained only locally.

In the following, we relax those assumptions following the strategy developed in the continuous case in \cite{Andrieu} and later in \cite{Andrieu_SICON_2014} and \cite{Bernard}. We require the system to be time reversible, and replace the observability hypothesis of the linearized system by a backward distinguishability hypothesis on the nonlinear system itself. In so doing, we obtain the existence and the injectivity (not only locally) of such a function of the state.

This paper is organized as follows. In the next part of the introduction (Section \ref{secps}), we state our problem in a more precise way and introduce some notations and definitions. We also prove a first result that guarantees the existence of an observer as soon as there exists a continuous uniformly injective map satisfying some functional equation. 
Our main results can be found in Section \ref{secres}. We state sufficient conditions for the existence, injectivity and also unicity of such a map.
We provide in Section \ref{secex} some examples and applications of those results.
we examine linear systems with polynomial output and also discrete-time systems that approximate continuous-time systems.

Throughout the paper,
we denote by $|\cdot|$ the usual Euclidean norm and by $\|\cdot\|$ the induced matrix norm.
% Some conclusions are given in Section \ref{secconclu}.

\subsection{Problem statement} \label{secps}

We consider the discrete-time system
\begin{equation}
    x_{k+1} = f(x_k),\qquad
    y_k = h(x_k),
\label{syst}
\end{equation}
with state $x\in\R^n$, output $y\in\R^p$ and suitable functions $f$ and $h$.
In this paper, we deal with the problem of existence of an observer for system \eqref{syst}.
We denote $X_k(x_0) = f^k(x_0)$ the value at time $k$ of the unique solution of system \eqref{syst} initialized at $x_0\in\R^n$, and $Y_k(x_0) = h(X_k(x_0))$ the corresponding output.
Let $\X_0 \subset \X \subset \R^n$  such that for all initial condition $x_0\in\X_0$ and all $k\in\N$, $X_k(x_0)\in\X$.
% Let $\X_0 \subset \R^n$  be a set of initial conditions. Let $\X\subset\R^n$ such that for any $x_0\in\X_0$ and any $k\in\N$, $X_k(x_0)\in\X$.
% \begin{definition}\label{defobs}
% Let $m$ be a positive integer.
% An observer for the system \eqref{syst} is a sequence of maps $\hat{X}_k: \X_0\times\R^m\to\R^n$ for $k\geq0$ such that
% \begin{enumerate}
%     \item there exists two functions $\phi : \R^n\times \R^p \to \R^n$, $\psi:\R^m\to\R^n$ such that for all $(x_0, \xi_0)\in\X_0\times\R^m$ and all $k\in\N$,
%     \begin{equation}
%         \xi_{k+1} = \phi(\xi_k, Y_k(x_0)),\quad
%         \hat{X}_k(x_0, \xi_0) = \psi(\xi_k).
%         \label{eqobs}
%     \end{equation}
%     \item for all $(x_0, \xi_0)\in\X_0\times\R^m$,
%     \begin{equation}
%         \lim_{k\to+\infty} |X_k(x_0) - \hat{X}_k(x_0, \xi_0)| = 0.
%         \label{convobs}
%     \end{equation}
% \end{enumerate}
% \end{definition}
% Note that, even if $\hat{X}_k$ seems to depend directly of $x_0$, it is actually not the case. As \eqref{eqobs} says, $\hat{X}_k$ depends only of the measurements $Y_0(x_0),Y_1(x_0),\dots,Y_{k-1}(x_0)$ through the dynamic of $(\xi_k)$. Also remark that the convergence given by \eqref{convobs} is independent of the choice of $x_0\in\X_0$ and of the initialisation $\xi_0$ of the observer.\\
\begin{definition}\label{defobs}
Let $m$ be a positive integer, $\phi : \R^n\times \R^p \to \R^n$ and $\psi:\R^m\to\R^n$.
The discrete-time dynamical system given by
\begin{equation}
    \xi_{k+1} = \phi(\xi_k, y_k),\qquad
    \hat{x}_k = \psi(\xi_k),
\label{systobs}
\end{equation}
is called an observer for \eqref{syst} if and only if,
for all $(x_0, \xi_0)\in\X_0\times\R^m$,
the solution of the coupled system \eqref{syst}-\eqref{systobs}, denoted by $(X_k(x_0), \hat{X}_k(x_0, \xi_0))_{k\geq0}$,
%the corresponding solution $(\hat{x}_k)_{k\geq0}$ denoted by $(\hat{X}_k(x_0, \xi_0))_{k\geq0}$
satisfies
\begin{equation}  
    \lim_{k\to+\infty} \big|X_k(x_0) - \hat{X}_k(x_0, \xi_0) \big| = 0.
    \label{convobs}
\end{equation}
\end{definition}
Note that, even if $\hat{X}_k$ seems to depend directly of $x_0$, it is actually not the case. As \eqref{systobs} says, $\hat{X}_k$ depends only of the measurements $Y_0(x_0),Y_1(x_0),\dots,Y_{k-1}(x_0)$ through the dynamic of $(\xi_k)_{k\geq0}$.\\

We follow the Luenberger-like methodology in order to design an observer for system  \eqref{syst}. Let $m$ be a positive integer.
First, we try to transform \eqref{syst} into
\begin{equation}
\xi_{k+1} = A\xi_k + B y_k.
\label{systxi}
\end{equation}
with $A\in\R^{m\times m}$ a matrix with spectral radius $\rho(A)<1$ and $B\in\R^{m\times p}$. In order to do this, we look for a continuous map $T: \X \to \R^m$ such that, for any $x_0\in\X_0$ and any $k\in\N$,
\begin{equation}
T(X_{k+1}(x_0)) = A T(X_k(x_0)) + B Y_k(x_0).
\label{eqT0}
\end{equation}
Let $\Xi_k(x_0, \xi_0)$ denote the value at time $k$ of the unique solution of system \eqref{systxi} with initial condition $\xi_0\in\R^m$ and measurements $y_k = Y_k(x_0)$.
Note that, for any $(x_0, \xi_0)\in\X_0\times\R^m$,
\begin{align}
\Xi_{k+1}(x_0, \xi_0) - T(X_{k+1}(x_0))
= A ( \Xi_k(x_0, \xi_0) - T(X_k(x_0)))
\label{eqdif}\end{align}
%\begin{equation}
%\Xi_{k+1}(x_0, \xi_0) - T(X_{k+1}(x_0)) = A ( \Xi_k(x_0, \xi_0) - T(X_k(x_0)))
%\label{eqdif}
%\end{equation}
and since $\rho(A)<1$, $ \Xi_k(x_0, \xi_0) - T(X_k(x_0))$ converges geometrically towards zero. Hence, implementing system \eqref{systxi}, one can deduce an approximation of $T(x_k)$ as $k$ goes to infinity.
Then, if $T$ is injective, one can estimate the state of system \eqref{syst}.
More precisely, we have the following theorem.
\begin{theorem}\label{thobs}
Let $m$ be a positive integer, $A\in\R^{m\times m}$ such that $\rho(A)<1$ and $B\in\R^{m\times p}$. Let $T: \X \to \R^m$ be a continuous map. Assume the following:
\begin{enumerate}
\item For all $x\in\X$, $T$ satisfies
\begin{equation}
T(f(x))=AT(x)+Bh(x).
\label{eqT}
\end{equation}
\item $T$ is uniformly injective, that is, there exists $\alpha$ a class $\K^\infty$ function such that for all $(x_1, x_2) \in\X^2$,
\begin{equation}
|x_1 - x_2| \leq \alpha(|T(x_1)-T(x_2)|.
\label{unifinj}
\end{equation}
\end{enumerate}
Then there exists a map $T^*: \R^m \to \R^n$ such that $(\hat{X}_k)_{k\geq0}$ defined by $\hat{X}_k(x_0, \xi_0) = T^*(\Xi_k(x_0, \xi_0))$ for all $(x_0, \xi_0)\in\X_0\times\R^m$ is the solution of an observer for \eqref{syst}.
\end{theorem}
\begin{proof}
Clearly, \eqref{eqT} implies that \eqref{eqT0} is satisfied for all $x_0\in\X_0$ and all $k\in\N$.
Let $(x_0, \xi_0)\in\X_0\times\R^m$.
Since $\rho(A)<1$, it follows from \eqref{eqdif} that
\begin{equation}
\lim_{k\to+\infty} \Xi_k(x_0, \xi_0) - T(X_k(\xi_0)) = 0.
\end{equation}
From the uniform injectivity of $T$, there exists a pseudo-inverse $T^{-1}:T(\X)\to\R^n$ such that for all $x$  in $\X$ $T^{-1}(T(x))=x$ and for all $(\xi_1, \xi_2) \in T(\X)^2$,
\begin{equation}
|T^{-1}(\xi_1) - T^{-1}(\xi_2)| \leq \alpha(|\xi_1 - \xi_2|).
\label{eqinvT}
\end{equation}
According to \cite[Theorem 2]{McShane}, there exists a function $T^*: \R^m \to \R^n$, that is an extension to $\R^m$ of $T^{-1}$, satisfying \eqref{eqinvT} for all $(\xi_1, \xi_2) \in (\R^m)^2$. Hence,
\begin{equation}
|T^*(\xi) - x| \leq \alpha(|\xi - T(x)|)\quad \forall \xi\in\R^m,\ \forall x\in\X.
\label{lemtech}
\end{equation}
Thus $|T^*(\Xi_k(x_0, \xi_0)) - X_k(\xi_0)| \to 0$ as $k$ goes to infinity.
Setting $\phi : (\xi, y)\in\R^n\times\R^p\mapsto A\xi+By$ and $\psi = T^*$, it follows from the Definition \ref{defobs} that $(\hat{X}_k)_{k\geq0}$ defined by $\hat{X}_k(x_0, \xi_0) = T^*(\Xi_k(x_0, \xi_0))$ is the solution of an observer for \eqref{syst}.
\end{proof}
Then it is sufficient to prove the existence of a uniformly injective continuous map $T: \X \mapsto \R^m$ satisfying \eqref{eqT} for some positive integer $m$ in order to design an observer for \eqref{syst}. In the next section, we state sufficient conditions for the existence, injectivity, and also unicity of a continuous map $T$ solution of \eqref{eqT}.
\begin{remark}
Note that if $\X$ is a compact subset of $\R^n$, then every continuous injective map $T:\X\to\R^m$ is also uniformly injective in the sense of \eqref{unifinj}.
In the following, we are interested in the injectivity of $T$.
If uniform injectivity is required (for example to apply Theorem \ref{thobs}), then one must either assume $\X$ compact or prove the uniform injectivity by other means.
\end{remark}

\section{Results and comments}\label{secres}

\subsection{Existence of the transformation}

First, we are interested in the existence of a map $T$ satisfying \eqref{eqT}.
In \cite{Andrieu}, V. Andrieu and L. Praly have proved the existence of a so-called Kazantzis--Kravaris/Luenberger observer for continuous-time systems of the form
\begin{equation}
    \dot x = f(x),\qquad
    y = h(x).
\label{continu}
\end{equation}
We follow the same methodology and adapt it in the discrete case. We need to make some assumptions on the system.
\begin{assumption}\label{hyp0}
$f$ is invertible and $f^{-1}$ and $h$ are continuous.
\end{assumption}
\begin{assumption}\label{hypconst}
There exist four non-negative constants $C_1$, $C_2$, $C_1'$ and $C_2'$ such that,
for all $x\in\R^n,$
\begin{equation}
    |x|\leq C_1 + C_2 |f(x)|, \qquad
    |h(x)|\leq C_1' + C_2' |x|.
\end{equation}
%\begin{align}
%    |x| \leq C_1 + C_2 |f(x)|,\qquad \forall x\in\R^n,\\
%    |h(x)| \leq C_1' + C_2' |x|,\qquad \forall x\in\R^n.
%\end{align}
\end{assumption}

\medskip
\begin{remark}\label{rmhyp}
Note that Assumptions \ref{hyp0} and \ref{hypconst} are satisfied in particular if $f$ is invertible and both $f^{-1}$ and $h$ are globally Lipschitz. We will use this remark in the next section about the injectivity of $T$.
\end{remark}
%%%%%%%%%%%%%%%%
%%%%%%%%%%%%%%%%
% Say that if $\X$ is compact and backward stable, then Assumption 2 is useless !!!
%%%%%%%%%%%%%%%%
%%%%%%%%%%%%%%%%

For all non-negative integer $i$, we denote $\circ$ the composition operator and
\begin{align*}
f^i = \underbrace{f\circ f\circ\dots\circ f}_{i \text{ times}},\qquad
f^{-i} = (f^{-1})^i.
\end{align*}

\begin{theorem}
\label{thexist}
Let $m$ be a positive integer, $A\in\R^{m\times m}$ a normal matrix such that $\rho(A)<\min\left\{1, {1}/{C_2}\right\}$ and $B\in\R^{m\times p}$.
Assume that Assumptions \ref{hyp0} and \ref{hypconst} are satisfied.
For all $x\in\X$, set
\begin{equation}
T(x) = \sum_{i=0}^{+\infty} A^i B h(f^{-(i+1)}(x)).
\label{defT}
\end{equation}
Then $T: \X \to \R^m$ is well defined, continuous, and satisfies \eqref{eqT}.
\end{theorem}
\begin{proof}
For all $x\in\X$ and all non-negative integer $i$, let $a_i(x) = A^i B h(f^{-(i+1)}(x))$.
According to Assumption \ref{hyp0}, each $a_i$ is continuous on $\X$.
Note that, since $A$ is normal, $\rho(A) = \norm{A}$. Then, according to Assumption \ref{hypconst}, we have for all $x\in\X$
\begin{align}
|a_i(x)| \leq 
\rho(A)^i \|B\|
\Bigg(C_1' + C_2'\Bigg(C_2^{i+1}|x| + C_1\sum_{j=0}^iC_2^j\Bigg)\Bigg).
\label{eqdom}
\end{align}
Since $\rho(A) < 1$ and $\rho(A)C_2<1$ the Lebesgue dominated convergence theorem applied on any compact set implies that \eqref{defT} defines a continuous function.
Moreover, for any $x\in\X$,
\begin{align*}
T(f(x))
% = T(f(x))
&= \sum_{i=0}^{+\infty} A^i B h(f^{-(i+1)}(f(x)))\\
% &= \sum_{i=0}^{+\infty} A^i B h(f^{-i}(x))\\
&= A \sum_{i=0}^{+\infty} A^{i-1} B h(f^{-i}(x))\\
&= A \sum_{i=0}^{+\infty} A^{i} B h(f^{-(i+1)}(x)) + Bh(x)\\
&= A T(x) + B h(x),
\end{align*}
which shows that $T$ satisfies \eqref{eqT}.
\end{proof}

\subsection{Injectivity with backward distinguishability}

In order to obtain that $T$ defined by \eqref{defT} is injective, we introduce the following \emph{backward distinguishability} assumption on the system.
\begin{assumption}
For all $(x_1, x_2)\in\X^2$, if $x_1\neq x_2$, then there exists a positive integer $i$ such that  $h(f^{-i}(x_1)) \neq  h(f^{-i}(x_2))$.
\label{hypback}
\end{assumption}
We also need stronger hypothesis on the system than in the previous section.
\begin{assumption}\label{hyp1}
$f$ is invertible and $f^{-1}$ and $h$ are of class $C^1$ and globally Lipschitz.
\end{assumption}
According to the Remark \ref{rmhyp}, if Assumption \ref{hyp1} holds, then Assumptions \ref{hyp0} and \ref{hypconst} are satisfied.
We denote by $I_k$ the identity $k\times k$ matrix, by $\otimes$ the Kronecker product and by $A^*$ the conjugate transpose matrix of $A$.

\begin{theorem}\label{thinj}
Let Assumptions \ref{hypback} and \ref{hyp1} hold.
Let $m = (n+1)p$ and $B = (1,\dots,1)^* \otimes I_p \in\C^{m\times p}$.
Let $C_2 = \sup\{|(f^{-1})'(x)|, x\in\X\}$ and
$\DD$ be the open disc of $\C$ of radius $\min\left\{1, {1}/{C_2}\right\}$.
Then there exists a subset $\RR \subset \DD^{n+1}$ of zero Lebesgue measure in $\C^{n+1}$ such that, for any
$(\lambda_1, \dots \lambda_{n+1}) \in \DD^{n+1}\setminus \RR$,
%$(\lambda_i)_{1\leq i\leq n+1} \in \DD^{n+1}\setminus \RR$,
the matrix $A = \diag(\lambda_1,\dots,\lambda_{n+1}) \otimes I_p \in\C^{m\times m}$ is such that the map $T:\X\to\C^m$ defined by \eqref{defT} is well-defined, of class $C^1$ and one-to-one.
\end{theorem}
\begin{proof}
Let  $(\lambda_1, \dots \lambda_{n+1}) \in \DD^{n+1}$ and $A = \diag(\lambda_1,\dots,\lambda_{n+1}) \otimes I_p \in\C^{m\times m}$. Let $T:\X\to\C^m$ be defined as in \eqref{defT}.
For all $\lambda \in\DD$, let
\begin{equation}
T_\lambda(x) = \sum_{i=0}^{+\infty} \lambda^i h(f^{-(i+1)}(x)),\quad \forall x\in\X.
\label{defTl}
\end{equation}
Let $a_i(x) = \lambda^i h(f^{-(i+1)}(x))$ for all $x\in\X$. Then each $a_i$ is of class $C^1$ on $\X$ by Assumption \ref{hyp1}, and we have the following domination:
\begin{equation*}
\left|a_i'(x)\right| \leq \lambda^i C_2' C_2^{i+1}
\end{equation*}
with $C_2' = \sup\{|h'(x)|, x\in\X\}$. Moreover, $\lambda C_2 < 1$.
So the Lebesgue dominated convergence theorem implies that for each $\lambda\in\DD$, $T_\lambda:\X\to\C^p$ is well-defined and of class $C^1$.
Considering the structure of $A$ and $B$, remark that up to a permutation of coordinates we have
\begin{equation*}
    T(x) = \left(T_{\lambda_1}(x),\dots,T_{\lambda_{n+1}}(x)\right)^*
\end{equation*}
It is sufficient to prove that $T:\X\to\C^p$ is one-to-one for almost all $(\lambda_1,\dots,\lambda_{n+1})\in\DD^{n+1}$.

In order to do this, we need the following lemma, established by L. Praly and V. Andrieu in \cite[Lemma 1]{Andrieu}, which is a modified version of \cite[Lemma 3.2]{Coron} due to J.-M. Coron.

\begin{lemma}\label{lemcoron}
Let $\DD$ and $\Gamma$ be open subsets of $\C$ and $\R^{2n}$, respectively. Let $g:\Gamma\times\DD\to\C^p$ be a function which is holomorphic in $\lambda$ for each $\xx\in\Gamma$ and $C^1$ in $\xx$ for each $\lambda\in\DD$.
% If for each pair $(\xx, \lambda)\in\Gamma\times\DD$ we can find, for at least one of the components $g_j$ of $g$, a non-negative integer $k$ satisfying
% \begin{equation*}
% \frac{\partial^kg_j}{\partial\lambda^k}(\xx, \lambda) \neq 0 \text{ and }
% \frac{\partial^ig_j}{\partial\lambda^i}(\xx, \lambda) = 0\
% \forall i\in\{0,\dots,k-1\},
% \end{equation*}
If for each $\xx\in\Gamma$, the function $\lambda\in\DD\mapsto g(\xx, \lambda)$ is not constantly zero,
then the set
\begin{align}
\RR = 
\bigcup_{\xx\in\Gamma} \Big\{
(\lambda_1,\dots,\lambda_{n+1}) \in \DD^{n+1}\ \big|\
\forall i\in\{1,\dots,n+1\},\ g(\xx, \lambda_i) = 0 \Big\}
\end{align}
has zero Lebesgue measure in $\C^{n+1}$.
\end{lemma}
% A proof of this lemma can be found in \cite{Andrieu}.
We apply this lemma to $\Gamma = \{(x_1, x_2)\in\X^2\ |\ x_1\neq x_2\}$ and $g = \Delta T$ defined as follows:
\begin{equation}
\Delta T : (x_1, x_2, \lambda)\in\X^2\times\DD \mapsto T_\lambda(x_1) - T_\lambda(x_2)
\label{defdt}
\end{equation}
Clearly, $\Delta T(x_1, x_2, \cdot)$ is holomorphic on $\DD$ for each $(x_1, x_2)\in\X^2$ and $\Delta T(\cdot, \lambda)$ is of class $C^1$ on $\X^2$ for each $\lambda\in\DD$. Fix $(x_1, x_2)\in\Gamma$. Now, we prove that $\Delta T(x_1, x_2, \cdot)$ is not identically zero on $\DD$. Assume the contrary. By unicity of the power series expansion, we get that for all positive integer $i$,
\begin{equation}
h(f^{-i}(x_1)) = h(f^{-i}(x_2))
\end{equation}
According to the \emph{backward distinguishability} Assumption \ref{hypback}, it implies that $x_1=x_2$ which is contradictory with the fact that $(x_1, x_2)\in\Gamma$. Hence, $\Delta T(x_1, x_2, \cdot)$ is not identically zero on $\DD$.

Since $\DD$ is a convex subset of $\C$ and $\Delta T(x_1, x_2, \cdot)$ is holomorphic, its zero are isolated and with finite multiplicity. Hence the hypotheses of Lemma \ref{lemcoron} are satisfied. Thus, $\RR\subset\DD^{n+1}$ has zero Lebesgue measure and for all $(\lambda_1,\dots,\lambda_{n+1})\in\DD^{n+1}\setminus \RR$, $T$ is injective by definition of $\Delta T$.
\end{proof}
\begin{remark}\label{rmreal}
The function $T$ and the matrices $A$ and $B$ defined Theorem \ref{thinj} take complex values while previous Theorems \ref{thobs} and \ref{thexist} remain in the real frame. However, one can choose two different ways to bridge this gap.
\begin{itemize}
    \item State Theorems \ref{thobs} and \ref{thexist} in the complex frame. The proofs remain identical. One should simply change the domains and codomains of $f$ and $h$.
    \item Instead of considering $A = \diag(\lambda_1,\dots,\lambda_{n+1}) \otimes I_p \in\C^{m\times m}$ and $B = (1,\dots,1)^* \otimes I_p \in\C^{m\times p}$, one should either consider $\tilde{A} = \diag(\Lambda_1,\dots,\Lambda_{n+1}) \otimes I_p \in\R^{2m\times 2m}$ and $\tilde{B} = (\mathbb{I},\dots,\mathbb{I})^* \otimes I_p \in\R^{2m\times p}$, where
    \begin{align*}
        \Lambda_i = \begin{pmatrix}
        \Re(\lambda_i) &  -\Im(\lambda_i)\\
        \Im(\lambda_i) & \Re(\lambda_i)
        \end{pmatrix},\qquad
        \mathbb{I} = \begin{pmatrix}
        1& 0
        \end{pmatrix}.
    \end{align*}
    Then for all real sequence of measurements $(y_k)_{k\geq0}$ the solutions of $\tilde{\xi}_{k+1} = \tilde{A}\tilde{\xi}_k+\tilde{B}y_k$ contain the real and imaginary parts of the solutions of $\xi_{k+1} = A\xi_k+By_k$.
\end{itemize}
\end{remark}

\subsection{Unicity}

One can also wonder in which cases does the unicity of $T$ satisfying \eqref{eqT} holds. More than a theoretical question, this fact may be useful in practice in order to obtain the injectivity of $T$. Most of the time, the function $T$ given by \eqref{defT} is difficult to compute. Since the matrix $A$ has spectral radius strictly inferior to $1$, an approximation of $T$ is given by
\begin{equation}
    T_N(x) = \sum_{i=0}^{N} A^i B h(f^{-(i+1)}(x)),\quad \forall x\in\X.
\end{equation}
for all $N\geq0$. Then $|T(x)-T_n(x)|\to 0$ as $N\to+\infty$. However, if $f$ and $h$ have more properties (for example if $f$ is linear and $h$ is polynomial, see Section \ref{seclinpol}), there may exist another solution $\tilde{T}$ of \eqref{eqT} much more easier to compute than $T$. Then, the question of the injectivity of that new $\tilde{T}$ remains open \emph{a priori}. But if \eqref{eqT} has a unique solution for $A$ and $B$ complex matrices chosen has in Theorem \ref{thinj}, then $T = \tilde{T}$ and hence $\tilde{T}$ is injective.
Now, we state our unicity theorem.

\begin{theorem}\label{thunic}
Let $m$ be a positive integer, $A\in\R^{m\times m}$ such that $\rho(A)<1$ and $B\in\R^{m\times p}$. Let Assumption \ref{hyp0} hold and make the following backward stability hypothesis on $\X$:
\begin{equation}
\forall x\in\X,\ \forall i\geq1,\quad f^{-i}(x) \in\X.
\label{backstab}
\end{equation}
Assume also that $\X$ is compact.
Then there exists one and only one continuous function $T: \X \to \R^m$ that satisfy \eqref{eqT} for all $x\in\X$.
\end{theorem}
\begin{proof}
First, we prove that the continuous solution of \eqref{eqT} is unique.
Let $T_1,\ T_2: \X \to \R^m$ be two continuous solutions of \eqref{eqT}.
Let $x\in\X$. Then for all $i\in\N$,
\begin{align*}
T_1(x) - T_2(x)
&= (T_1-T_2)(f^i(f^{-i}(x)))\\
&= A^i (T_1-T_2)(f^{-i}(x)). \tag{from \eqref{eqT}}
\end{align*}
Since $\X$ is compact, satisfy \eqref{backstab} and $T_1$ and $T_2$ are continuous, there exists a constant $K>0$ such that
$|(T_1-T_2)(f^{-i}(x))| \leq K$ for all $i\in\N$.
Since moreover $\rho(A)<1$, $A^i (T_1-T_2)(f^{-i}(x)) \to 0$ as $i\to+\infty$.
Thus $T_1(x) - T_2(x) = 0$.

The existence of a continuous $T$ satisfying \eqref{eqT} follows from the Theorem \ref{thexist} and from the fact that Assumption \ref{hypconst} can be replaced in its proof by the fact that $\X$ is compact and backward stable\footnote{
Similarly, using the same trick, one can easily show that the hypothesis of \emph{globally Lipschitz} in Assumption \ref{hyp1} can be replaced in the proof of Theorem \ref{thinj} by the fact that $\X$ is compact and backward stable.
}. Indeed, the series \eqref{defT} still defines a continuous function since the domination
\begin{equation}
|a_i(x)| \leq \rho(A)^i ||B|| \sup_{\tilde{x}\in\X} h(\tilde{x})
\end{equation}
holds for all $x\in\X$ and can replace \eqref{eqdom}. Then one may apply the Lebesgue dominated convergence on $\X$.
\end{proof}

To conclude this section, recall that we have now at our disposal three theorems that ensures under different conditions on \eqref{syst} the existence, unicity and injectivity of a continuous map $T$ satisfying \eqref{eqT}.
In the next section, we illustrate on examples how to use those tools. In particular, we study systems with linear dynamics and polynomial output, and emphasize the link between the Luenberger observers developed in \cite{Andrieu} for continuous-time systems and the discrete-time observers developed in this paper for theirs first-order approximations.

\section{Examples}\label{secex}

\subsection{Linear dynamics with polynomial output}\label{seclinpol}

We consider first the system with linear dynamic and polynomial output of degree $d$
\begin{equation}
x_{k+1} = F x_k,\quad y_k = H P_d(x)
\end{equation}
with $P_d:\R^n\to\R^{k_d}$ a vector containing the $k_d$ possible monomials with degree less or equal than $d$, $F\in\R^{n\times n}$ and $H\in\R^{p \times k_d}$.
Then we have the following proposition.
\begin{proposition}
Let $m$ be a positive integer and $B\in\R^{m\times p}$.
There exists a subset $\SS$ of zero Lebesgue measure in $\R^{m\times m}$ such that for all $A \in\R^{m\times m}\setminus\SS$, there exists a function $T:\R^n\mapsto\R^m$
of the form
\begin{equation}
    T(x) = M P_d(x), \qquad \forall x\in\R^n
\label{Tpoly}
\end{equation}
for some $M\in\R^{m\times k_d}$, that satisfies \eqref{eqT} for any $x\in\R^n$.
\end{proposition}
\begin{proof}
First, note that since $P_d(Fx)$ is a vector containing polynomials of $x$ with degree inferior to $d$, there exists a matrix $D\in\R^{k_d\times k_d}$ such that
\begin{equation}
P_d(Fx) = DP_d(x),\qquad \forall x\in\R^n.
\label{eqD}
\end{equation}
% Let $A \in\R^{m\times m}$.
Since the set of eigenvalues of $D$ is finite, the spectra of $D$ and $-A$ are disjoint for almost all $A\in \R^{m\times m}$ \ie there exists a subset $\SS\subset\R^{m\times m}$ of zero Lebesgue measure such that the spectra of $D$ and $-A$ are disjoint for all $A\in \R^{m\times m}\setminus\SS$.
For such matrices $A$ the Sylvester equation
\begin{equation}
M D = A M + BH
\label{sylv}
\end{equation}
has a unique solution $M\in\R^{m\times k_d}$. Set $T$ as in \eqref{Tpoly}. It remains to check that \eqref{eqT} is satisfied for $f=F$ and $h=HP_d$. For all $x\in\R^n$,
\begin{align}
T(Fx)
&= M P_d(Fx)\tag{from \eqref{Tpoly}}\\
&= M D P_d(x)\tag{from \eqref{eqD}}\\
&= A M P_d(x) + BH P_d(x) \tag{from \eqref{sylv}}\\
&= AT(x) + BHP_d(x).\nonumber
\end{align}
\end{proof}
\begin{remark}
Note that the result is still true if $A$ and $B$ are complex matrices. Then $T$ takes complex values. The proof remains identical.
\end{remark}
\begin{remark}
Choose a set $\X_0\subset\R^n$ of initial condition and let $\X$ be as usual such that $X_k(x_0)\in\X$ for all $x_0\in\X_0$ and all $k\in\N$.
Note that if $F$ is invertible and if $\X$ is compact and backward stable, then the assumptions of the Theorem \ref{thunic} hold.
Assume also that Assumptions \ref{hypback} and \ref{hyp1} hold and apply Theorem \ref{thinj} with $m=(n+1)p$.
Then, for almost all $(\lambda_1, \dots,\lambda_{n+1})\in\C^{n+1}$,
and for complex matrices $A$ and $B$ as in Theorem \ref{thinj}, we have
\begin{equation}
T(x) = MP_d(x) = \sum_{i=0}^{+\infty} A^i B h(f^{-(i+1)}(x))
\end{equation}
for all $x\in\X$.
In particular, $T$ defined by \eqref{Tpoly} is injective.

\end{remark}

\subsection{Link with the continuous Luenberger observer}

In this section, we are interested in the link between the continuous Luenberger observer developed in \cite{Andrieu} for system \eqref{continu} and the discrete observer developed in the previous sections for a discrete-time version of \eqref{continu}.
% In particular, we show on an example that the stationary map $T$ Luenberger 

\subsubsection{Continuous-time system}

We consider the following example with linear dynamic and polynomial output:
\begin{equation}
    \begin{cases}
    \dot x_1 = x_2\\
    \dot x_2 = - x_1
    \end{cases},\quad
    y = x_1^2 - x_2^2 + x_1 + x_2.
\label{continu2}
\end{equation}
It can be shown that this system is weakly differentially observable\footnote{
First, the map $(x_1^2-x_2^2, x_1+x_2) \mapsto (y, \ddot y)$ is injective.
Similarly, $(x_1x_2, x_1-x_2) \mapsto (\dot y, \dddot{y\hspace{0pt}})$ is also injective.
Combining those results, we get that $(x_1, x_2) \mapsto (y, \dot y, \ddot y, \dddot{y\hspace{0pt}})$ is injective.
}
of order $4$ on $\R^2$ in the sense of \cite[Definition 1]{Bernard}.
Following \cite{Bernard}, we seek $T_\lambda:\R\to\R^n$ such that
\begin{equation}
\frac{\dd}{\dd t}T_\lambda(x) = \lambda T_\lambda(x) + y
\label{eqTcont}
\end{equation}
for some $\lambda<0$. Since \eqref{continu2} has linear dynamic and polynomial output of degree 2, one can look for $T$ of the form
\begin{equation}
    T_\lambda(x) = x^*\begin{pmatrix}
    a& {c}/{2}\\
    {c}/{2}&b
    \end{pmatrix}x
    + \begin{pmatrix}d&e\end{pmatrix}x
\end{equation}
for some $(a, b, c, d, e)\in\R^5$. Then \eqref{eqTcont} holds if and only if
\begin{gather}
    % \begin{cases}
    -c = \lambda a + 1,\quad%\\
    c = \lambda b-1,\quad%\\
    2(a-b) = \lambda c,\nonumber\\
    -e = \lambda d + 1,\quad
    d = \lambda e + 1.
    % \end{cases}
\label{eqabccont}
\end{gather}
The only solution of this equation is
\begin{gather}
    % \begin{cases}
    a = -\frac{\lambda}{4+\lambda^2},\quad%\\
    b = \frac{\lambda}{4+\lambda^2},\quad%\\
    c = -\frac{4}{4+\lambda^2},\nonumber\\
    d = \frac{1-\lambda}{1+\lambda^2},\quad
    e = -\frac{1+\lambda}{1+\lambda^2}.
    % \end{cases}
\label{eqabccont2}
\end{gather}
Since $T_\lambda$ is stationary, one could believe that this function provide an observer that could be efficient even for a numerical approximation of \eqref{continu2}.
However, as we will see in the following, it is not the case: for a given discrete approximation of \eqref{continu2}, it is better to design an observer based on the discrete-time system rather than to use the one given by $T_\lambda$.

\subsubsection{Associated first-order discrete-time system}

For some discretization parameter $\dt>0$, the associated first-order approximation\footnote{
Since \eqref{continu2} is weakly differentially observable, it can be shown that \eqref{disc2} is backward distinguishable as soon as $\dt$ is small enough.
} of \eqref{continu2} is
\begin{equation}
\begin{cases}
x_1(k+1) = x_1(k) + \dt x_2(k)\\
x_2(k+1) = x_2(k) - \dt x_1(k)\\
y_k = x_1(k)^2-x_2(k)^2+x_1(k)+x_2(k)
\end{cases}.
\label{disc2}
\end{equation}
We seek a function $T^d_\lambda:\R\to\R^n$ satisfying a first-order approximation of \eqref{eqTcont} given by the Euler explicit method:
\begin{equation}
    T^d_\lambda(x(k+1)) = (1+\lambda \dt) T^d_\lambda(x(k)) + \dt y_k.
    \label{eqTdisc}
\end{equation}
Since $\lambda<0$, it is sufficient to choose $\lambda\dt>-2$ to have $-1<1+\lambda \dt<1$. Now, we seek $T^d_\lambda$ of the form
\begin{equation}
    T^d_\lambda(x) = x^*\begin{pmatrix}
    a'& {c'}/{2}\\
    {c'}/{2}&b'
    \end{pmatrix}x
    + \begin{pmatrix}
    d'&e'
    \end{pmatrix}x
\end{equation}
for some $(a', b', c', d', e')\in\R^5$. Then \eqref{eqTdisc} holds if and only if $(d', e')$ satisfy the same equation that $(d, e)$ in \eqref{eqabccont} and $(a', b', c')$ satisfy
\begin{equation}
\left\{
    \begin{aligned}
    &- c' + b'\dt = \lambda a' + 1,\\
    &c' + a'\dt = \lambda b' - 1,\\
    &2(a'-b') - c'\dt = \lambda c'.\\
    \end{aligned}
\right.
\label{eqabcdisc}
\end{equation}
Remark that this equation is the same than \eqref{eqabccont2} when $\dt = 0$. This is coherent with the fact that \eqref{disc2} is a discretization of \eqref{continu2}.
Then, the only solution of \eqref{eqabcdisc} is such that $(d', e') = (d, e)$ for all $\dt>0$ and $(a', b', c')$ converges to $(a, b, c)$ as $\dt$ goes to 0:
\begin{equation}
\left\{
    \begin{aligned}
    \displaystyle
    a' &= -\frac{\lambda+\dt}{4+(\lambda+\dt)^2},\\
    \displaystyle
    b' &= \phantom{-}\frac{\lambda+\dt}{4+(\lambda+\dt)^2},\\
    \displaystyle
    c' &= -\frac{4}{4+(\lambda+\dt)^2}.
    \end{aligned}\right.
\label{eqabcdisc2}
\end{equation}
For $\dt>0$, the discrete observer given by $T^d_\lambda$ is therefore different from the continuous observer given by $T_\lambda$, even if their difference goes to 0 as $\dt$ goes to 0.

\subsubsection{Comparison of the observers}

Consider a numerical simulation of the continuous-time system \eqref{continu2} obtained by the Euler explicit first-order method, which corresponds to the discrete-time system \eqref{disc2}.
Then the map $T^d_\lambda$ given by \eqref{eqTdisc} is much more adapted to the design of a numerically efficient observer than the function $T_\lambda$ given by \eqref{eqTcont} that has been designed for \eqref{continu2}. More generally, in order to implement an observer for a continuous-time varying system, it is better to develop a discrete-time observer based on the numerical approximation of the system, rather than a continuous-time observer based on the original system itself.
% This phenomenon is well known in control theory, and has been studied in \cite{}?????????. We recover, by duality, the equivalent phenomenon about observability.\\

In order to highlight numerically this fact, we simulate the system \eqref{continu2} thanks to \eqref{disc2} and compare the accuracy of two observers: one based on functions of the form $T^d_\lambda$, and another based on functions of the form $T_\lambda$.
To obtain the observers, we fix $\dt>0$ and three arbitrary values $\lambda_i < 0$ satisfying $\lambda_i \dt>-2$ and use the fact that
\begin{equation}
    %\underbrace{
    \begin{pmatrix}
    1 & 0 & 1 & 1\\
    a_1 & c_1 & d_1 & e_1\\
    a_2 & c_2 & d_2 & e_2\\
    a_3 & c_3 & d_3 & e_3
    \end{pmatrix}
    %}_{=:J}
    \begin{pmatrix}
    x_1^2 - x_2^2\\
    x_1 x_2\\
    x_1\\
    x_2
    \end{pmatrix}
    =
    \begin{pmatrix}
    y\\
    T_1(x)\\
    T_2(x)\\
    T_3(x)
    \end{pmatrix}.
    \label{eqmat}
\end{equation}
where $(a_i, c_i, d_i, e_i)$ is given by \eqref{eqabccont2} (resp. \eqref{eqabcdisc2}) with $\lambda=\lambda_i$ and $T_i = T_{\lambda_i}$ (resp. $T_i = T^d_{\lambda_i}$).
Fix the following parameters and initial conditions:
\begin{equation}
    \dt = 0.01,\ x(0) = (1, 0),\ \lambda_i = -10\times i,\ \xi^i(0) = 0.
\end{equation}
Then the $4\times4$ matrix defined in \eqref{eqmat} is invertible. Hence one can reconstruct an approximation $(\hat{x}_1, \hat{x}_2)$ of the state $(x_1, x_2)$ from the measurement $y$ and approximations of $T_i(x)$ given by the dynamic
$\xi^i_{k+1} = (1+\lambda_i\dt)\xi^i_k + \dt y_k$.\\

\begin{figure}
    \centering
    \epsfbox{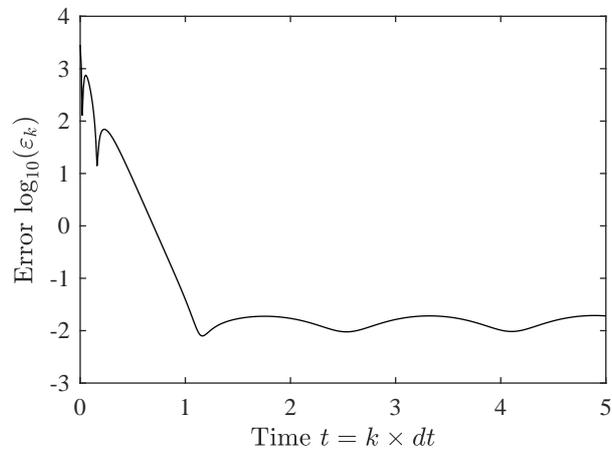}
    \caption{Evolution of the error between the state and the observer based on $T_{\lambda_i}$ in semi-log scale}      
    \label{figTc}
\end{figure}
\begin{figure}
    \centering
    \epsfbox{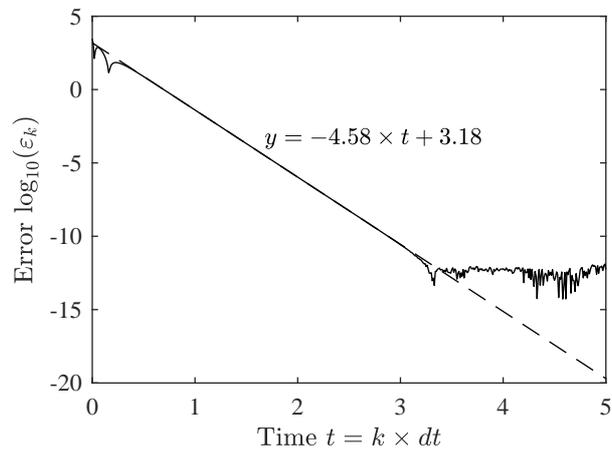}
    \caption{Evolution of the error between the state and the observer based on $T^d_{\lambda_i}$ in semi-log scale}
    \label{figTd}
\end{figure}

On Fig. \ref{figTc}, we plot on a semi-log scale the evolution of the absolute error $\varepsilon_k = \vert x_k-\hat x_k\vert$ between the state and its observer for $k \in \{0,\dots,500\}$ (\ie $t\in[0, 5]$) for the observer based on functions $T_{\lambda_i}$ designed for the original continuous-time system. Similarly, we make on Fig. \ref{figTd} the same plot but for the observer based on functions $T^d_{\lambda_i}$ designed for the discrete-time system. We clearly see that the observer based on $T^d_{\lambda_i}$ is much more efficient than the one based on $T_{\lambda_i}$.
On one hand, using $T^d_{\lambda_i}$, the error go to zero until it achieve $10^{-12}$, which is close to the machine epsilon $(\approx 10^{-16})$.
Moreover, the state observer seems to converge exponentially to the state, with a rate $r\approx -4.58$ (estimation based on a linear regression made on $[0.5, 3]$).
On the other hand, with $T_{\lambda_i}$, the observer does not converge to the state: it keeps an absolute error oscillating around $10^{-2}$.
This phenomenon is due to the fact that the trajectory of \eqref{disc2} is not invariant for this observer: even if it is well initialized (\ie $x(0) = \hat x(0)$), the observer will oscillate around the state.

% \section{CONCLUSIONS AND FUTURE WORKS}\label{secconclu}
\section{Conclusion}\label{secconclu}

We have shown how the initial Luenberger methodology can be applied to nonlinear discrete-time systems. It is based on the existence of a map satisfying some functional equation linked to the system, that transform the original system into a linear asymptotically stable one fed by the output. As soon as this map is uniformly injective, it allows us to estimate the state of the nonlinear system by simulating an autonomous system fed by the output and inverting this map. We stated sufficient conditions for the existence of such a map. In particular, we need the system to be reversible in time. Under a backward distinguishability hypothesis, we also proved that this map is injective.

% \subsection{Future Works}

% Continuous dynamic, discrete output.

% With control : time varying $T$

%%%%%%%%%%%%%%%%%%%%%%%%%%%%%%%%%%%%%%%%%%%%%%%%%%%%%%%%%%%%%%%%%%%%%%%%%%%%%%%%
% \section{ACKNOWLEDGEMENTS}
\bibliographystyle{plain}

\bibliography{references}

\end{document}